\documentclass[english,11pt]{amsart}
\usepackage{amsopn,amsthm,amsfonts,amsmath,amssymb,amscd}
\usepackage{babel}
\usepackage{nicematrix}
\usepackage{mathtools}
\usepackage{tikz}
\usetikzlibrary{arrows.meta}
\usetikzlibrary{calc}

\newtheorem{Theorem}{Theorem}[section]
\newtheorem{Lemma}[Theorem]{Lemma}
\newtheorem{Proposition}[Theorem]{Proposition}
\newtheorem{Corollary}[Theorem]{Corollary}
\newtheorem{Example}[Theorem]{Example}

\newtheorem{Def}[Theorem]{Definition}

\newcommand{\ZZ}{\mathbb{Z}}

\newcommand{\id}{{\rm id}}
\newcommand{\Aut}{{\mathrm Aut}}
\newcommand{\Inn}{{\mathrm Inn}}
\newcommand{\im}{{\rm im}}

\begin{document}

\title{The Cayley graph of a quandle}

\author{David Dol\v zan, Bogdana Oliynyk}

\address{D.~Dol\v zan:~Department of Mathematics, Faculty of Mathematics
and Physics, University of Ljubljana, Jadranska 19, SI-1000 Ljubljana, Slovenia, and Institute of Mathematics, Physics and Mechanics, Jadranska 19, SI-1000 Ljubljana, Slovenia; e-mail: 
david.dolzan@fmf.uni-lj.si}
\address{B.~Oliynyk; Faculty of Applied Mathematics, Silesian University of Technology, Kaszubska 23, 44-100, Gliwice, Poland. e-mail: bogdana.oliynyk@polsl.pl}


\thanks{The first author acknowledges the financial support from the Slovenian Research Agency  (research core funding No. P1-0222)}

\begin{abstract}
In this paper, we investigate structural properties of the Cayley graph of a quandle and describe this graph for several important classes of quandles, including conjugation, Takasaki, dihedral, and Alexander quandles. In particular, we prove that for an Alexander quandle $A_t(G)$ over a finite abelian group $G$, the connected components of the Cayley graph correspond to the cosets of the subgroup $\mathrm{im}(\mathrm{id}-t)$. We also show that the Cayley graphs of generalized Alexander quandles are regular. When the defining automorphism is inner, we give an explicit description of the forward orbits and prove that the connected components correspond to cosets of the subgroup generated by commutators with the defining element.
\end{abstract}

\maketitle 

 \section{Introduction}

\bigskip

Quandles are algebraic structures designed to encode the formal properties of conjugation in groups and, simultaneously, the Reidemeister moves in knot theory. Since their introduction, they have become a central tool in the algebraic study of knots and links, giving rise to computable invariants such as coloring invariants, cocycle invariants, and homological constructions. Beyond knot theory, quandles appear naturally in group theory, Hopf algebraic settings, set-theoretic solutions of the Yang--Baxter equation, and the study of symmetric spaces.

Historically, the algebraic structure underlying knot colorings can be traced back to the work of Joyce and Matveev in the early 1980s. Joyce introduced the term \emph{quandle} and developed the fundamental quandle of a knot as a complete invariant up to orientation reversal in a suitable sense \cite{Joyce82}. Independently, Matveev introduced essentially the same structure under the name \emph{distributive groupoid} (or \emph{kei} in the involutory case) \cite{Matveev82}. The connection to classical knot invariants is strong: the fundamental quandle can be viewed as an analogue of the fundamental group, with the advantage that its defining relations mirror the Reidemeister moves directly.

A key feature of quandle theory is that it provides an algebraic framework in which the arcs of a knot diagram may be colored by elements of a fixed quandle, with the crossing relations corresponding exactly to the quandle operation. This leads to a simple and computable knot invariant: the number of quandle colorings of a diagram is invariant under Reidemeister moves, hence is a knot invariant. These ideas have been expanded substantially, particularly through the development of quandle homology and quandle cohomology, culminating in state-sum cocycle invariants introduced by Carter, Jelsovsky, Kamada, Langford, and Saito \cite{CarterEtAl99,CarterEtAl03}. For general background on racks and quandles and their applications to knot invariants, see the monographs and surveys \cite{ElhamdadiNelson15,Kamanda17}.

We recall the standard definitions and fix notation. 

\begin{Def}
A \emph{quandle} is a set $Q$ equipped with a binary operation
\[
\rhd \colon Q \times Q \to Q,\qquad (x,y)\mapsto x \rhd y,
\]
such that for all $x,y,z \in Q$:
\begin{enumerate}
  \item \textbf{Idempotency:} $x \rhd x = x$.
  \item \textbf{Right invertibility:} for each $y\in Q$, the map $R_y\colon Q\to Q$ defined by $R_y(x)=x\rhd y$ is a bijection.
  \item \textbf{Self-distributivity:} $(x\rhd y)\rhd z = (x\rhd z)\rhd (y\rhd z)$.
\end{enumerate}
\end{Def}

If only axioms (2) and (3) are required, one obtains a \emph{rack}. Quandles can thus be viewed as idempotent racks. The bijections $R_y$ are called the \emph{right translations} of the quandle. Axiom (3) is frequently called \emph{right self-distributivity} and is the algebraic shadow of the third Reidemeister move.

A quandle $Q$ is called \emph{involutory} (or a \emph{kei}) if in addition
\[
(x \rhd y)\rhd y = x \qquad \text{for all } x,y\in Q.
\]

Cayley graphs provide a standard bridge between algebraic structure and combinatorial geometry. 
They were introduced by Cayley in the context of groups as a way to represent algebraic relations by a directed graph, and they have since become a fundamental tool in combinatorial and geometric group theory; see, for example, \cite{Cayley78,deLaHarpe00,GodsilRoyle01}. 
In the classical setting, the Cayley graph of a group $G$ with respect to a generating set $S$ encodes the action of $S$ by right multiplication, and many algebraic properties of $G$ become visible as graph-theoretic features of the resulting regular graph.

This construction extends naturally to other algebraic systems whose operations determine families of permutations. 
In particular, properties of the labeled Cayley digraphs of racks and quandles were investigated in \cite{LucTa}.

In this paper, we define the Cayley graph of a quandle as follows.

\begin{Def}
Let $Q$ be a quandle. The \emph{(right) Cayley graph} of $Q$ is the directed graph $\Gamma_Q$ defined as follows:
\begin{itemize}
  \item the vertex set is $V(\Gamma_Q) = Q$;
  \item for each $x,y \in Q$ there is a directed edge
  \[
  x \rightarrow x \rhd y.
  \]
\end{itemize}
\end{Def}

Cayley graphs allow one to phrase structural properties of quandles in purely graph-theoretic language. 
The structure of the paper is as follows. In Section~2, we recall the basic definitions and present several fundamental examples of quandles that will be used throughout the paper. In Section~3, we investigate structural properties of the Cayley graph of a quandle and describe its form for several important classes of quandles, including conjugation quandles, Takasaki quandles, dihedral quandles, and Alexander quandles. In particular, we prove that for an Alexander quandle $A_t(G)$ over a finite abelian group $G$, the connected components of the Cayley graph correspond to the cosets of the subgroup $\mathrm{im}(\mathrm{id}-t)$. In Section~4, we study Cayley graphs of generalized Alexander quandles. We show that these graphs are regular and, when the defining automorphism is inner, we obtain an explicit description of the forward orbits and prove that the connected components correspond to cosets of the subgroup generated by commutators with the defining element.

\medskip

\bigskip

 \section{Definitions and examples}
\bigskip

We begin with several important examples of quandles.
\begin{Example}[Trivial quandle]
Let $X$ be a nonempty set. Define a binary operation
\[
\rhd \colon X \times X \to X
\]
by
\[
x \rhd y = x \qquad \text{for all } x,y \in X.
\]
Then $(X,\rhd)$ is a quandle, called the \emph{trivial quandle} on $X$.
\end{Example}

Quandles arise naturally from group theory. Let $G$ be a group.

\begin{Example}[Conjugation quandle]
The set $G$ becomes a quandle with operation
\[
x \rhd y = y^{-1}xy.
\]
Idempotency follows from $x \rhd x = x^{-1}xx = x$, right invertibility follows because conjugation is an automorphism, and self-distributivity follows from the homomorphism property of conjugation. Any union of conjugacy classes in $G$ is a subquandle.
\end{Example}

\begin{Example}[Core quandle]
The set $G$ becomes a quandle with operation
\[
x \rhd y = yx^{-1}y.
\]
This quandle is involutory. It is sometimes called the \emph{core} of the group.

In case $G$ is abelian, this quandle is often called the \emph{Takasaki} quandle.
\end{Example}

\begin{Example}[Alexander quandle]
Let $G$ be an additive abelian group and $t \in \Aut(G)$. Then $G$ is a quandle under
\[
x \rhd y = t(x) + (\id_G-t)(y).
\]
This is the prototypical example of an affine quandle and plays an important role in the relationship between quandle invariants and classical Alexander-type invariants. Note that if $t(x)=-x$, then the Alexander quandle is exactly the Takasaki quandle. We denote the Alexander quandle  with respect to $t$ by $A_t(G)$.  
\end{Example}

\begin{Example}[Generalized Alexander quandle]
The preceding example can be generalized. Let $G$ be a group and $\varphi \in \Aut(G)$. Then $G$ is a quandle under
\[
x \rhd y = \varphi(xy^{-1})y.
\]
This is called a generalized Alexander quandle of $G$ with respect to $\varphi$. When 
$G$ is abelian, this construction coincides with the usual Alexander quandle.
\end{Example}

\begin{Example}[Dihedral quandle]
For $n\ge 2$, the set $\mathbb{Z}/n\mathbb{Z}$ with operation
\[
a \rhd b = 2b-a
\]
is an involutory quandle, denoted $R_n$. It corresponds to reflections in the dihedral group and yields the classical Fox $n$-colorings of knots.
\end{Example}

The next lemma is immediate.

\begin{Lemma}
   Let $Q$ be a quandle and $x \in Q$. Then there is a loop at $x$ in $\Gamma_Q$.
\end{Lemma}

Throughout the paper, all graphs are directed graphs with loops and without multiple edges. Thus, for any ordered pair $(a,b) \in V(\Gamma)\times V(\Gamma)$, there is at most one edge
from $a$ to $b$. Accordingly, we define a \emph{graph} $\Gamma$ as an ordered pair $\Gamma = (V(\Gamma), E(\Gamma))$, where $V(\Gamma)$ is the set of vertices and $E(\Gamma)$ is the set of edges (or arcs),
with $E(\Gamma) \subseteq V(\Gamma) \times V(\Gamma).$
Each edge $e \in E(\Gamma)$ starts at a vertex denoted by $i(e) \in V(\Gamma)$ and terminates at a vertex denoted by $t(e) \in V(\Gamma)$.
If $i(e) = t(e)$, then $e$ is called a \emph{loop} at the vertex $i(e)$. In particular, a loop at a vertex $a \in V(\Gamma)$ is identified with the pair $(a,a)$.

We extend the standard definition of simple graphs to the case of directed graphs.
We say that a directed graph $\Gamma$ is \emph{complete} if, for any two distinct vertices
$a,b \in V(\Gamma)$, there are directed edges from $a$ to $b$ and from $b$ to $a$ in $E(\Gamma)$, and if there is exactly one loop at each vertex of $\Gamma$. Analogously to the undirected case, we denote the complete directed graph with $n$ vertices by $K_n$.

We say that a directed graph $\Gamma$ is \emph{edgeless} if the edge set $E(\Gamma)$
consists only of loops at vertices, that is,
\[
E(\Gamma)=\{aa : a \in V(\Gamma)\}.
\] 

A directed graph $\Gamma$ is called \emph{symmetric} if, for all $a,c\in V(\Gamma)$, whenever there is an edge from $a$ to $c$, there is also an edge from $c$ to $a$. Throughout the paper, a directed graph is called connected if for every ordered pair of vertices $(u,v)$ there exists a directed path from $u$ to $v$.

We refer the reader to \cite{Diestel} for general background and for all undefined notions on graph theory used in the paper.

\bigskip

 \section{Some properties of Cayley graphs}
\bigskip

In this section, we investigate structural properties of the Cayley graph of a quandle and describe this graph for several important classes of quandles.
We start with the following simple lemma.

\begin{Lemma}
   Let $Q$ be a quandle. Then $\Gamma_Q$ is edgeless if and only if $Q$ is a trivial quandle.
\end{Lemma}
\begin{proof}
   $Q$ is a trivial quandle if and only if $x \rhd y = x \, \text{for all } x,y \in Q$, which is equivalent to saying that the only edges of $\Gamma_Q$ are the loops.
\end{proof}

\begin{Proposition}
 \label{conj}
Let $G$ be a group, and let $\mathrm{Conj}(G)$ denote the conjugation quandle of $G$. Then the following statements hold:
\begin{enumerate}
    \item 
    \label{conj1}
    The graph $\Gamma_{\mathrm{Conj}(G)}$ is symmetric.

    \item
    The connected components of $\Gamma_{\mathrm{Conj}(G)}$ are complete subgraphs induced by the conjugacy classes of $G$. 
        
\end{enumerate}
\end{Proposition}

\begin{proof}
{\bf 1.} Let $ac$ be an edge of the graph $\Gamma_{\mathrm{Conj}(G)}$.
By the definition of the Cayley graph and the conjugation quandle,
this means that there exists an element $b \in G$ such that $a \rhd b = c$, that is, $c=b^{-1}ab$. Equivalently,  $$a=bcb^{-1}=(b^{-1})^{-1}cb^{-1}.$$ Hence there is also an edge $ca \in E(\Gamma_{\mathrm{Conj}(G)})$ and therefore the graph   $\Gamma_{\mathrm{Conj}(G)}$ is symmetric. 

{\bf 2.} If there is an edge $ac$ in the graph $\Gamma_{\mathrm{Conj}(G)}$, then there exists
an element $b \in G$ such that $a \rhd b = c$, that is, $c = b^{-1}ab$. Hence, the elements $a$ and $c$ belong to the same conjugacy class of $G$.

Now let $a$ and $c$ be conjugate elements of $G$, that is, there exists an element $b \in G$ such that $c = b^{-1}ab$.
By the definition of the conjugation quandle, we have $a \rhd b = c$.
Therefore, there is an edge $ac \in E(\Gamma_{\mathrm{Conj}(G)})$.
Moreover, for any two conjugate elements $a$ and $c$ of the group $G$,
there is an edge between them in $\Gamma_{\mathrm{Conj}(G)}$.
Therefore, each conjugacy class induces a complete subgraph and these are precisely the connected components.
\end{proof}

\begin{Proposition}
\label{Takasaki}
  Let $T(\mathbb{Z})$ be the Takasaki quandle of the group  $\mathbb{Z}$. Then the following statements hold:
\begin{enumerate}
\item 
 The graph $\Gamma_{T(\mathbb{Z})}$ is symmetric.
 \item 
The graph  $\Gamma_{T(\mathbb{Z})}$ has two connected components, each isomorphic to the complete directed graph on a countably infinite vertex set, and for any $a,c \in T(\mathbb{Z})$ there is an edge from $a$ to $c$ if and only if $c+a \equiv 0 \pmod 2.$
\end{enumerate} 
\end{Proposition}

\begin{proof}
{\bf 1.} For $T(\mathbb Z)$ we have $a \rhd b = 2b-a$.
If $a \rhd b = c$, then $c=2b-a$, hence $a=2b-c$, i.e. $c \rhd b = a$.
Thus $\Gamma_{T(\mathbb Z)}$ is symmetric. 

{\bf 2.}
Let $a,c, b \in \mathbb{Z}$. For the Takasaki quandle $T(\mathbb{Z})$ the equality $a \rhd b = c$ means $c=2b-a$ or 
\begin{equation}
\label{takeq}
c+a=2b
\end{equation}
If $a$ and $c$ have the same parity, then for any $a$ and $c$ there exists exactly one $b \in \mathbb{Z}$ such that the equality \eqref{takeq} holds. So, there are edges $ca$ and $ac$ in the graph $\Gamma_{T(\mathbb{Z})}$. If $a$ is odd and $c$ is even or vice versa, such $b$ does not exist. This completes the proof of the proposition.
\end{proof}

\begin{Proposition}
\label{Prim_Dihed}
  Let $R_n$ be a dihedral quandle. Then the following statements hold:
\begin{enumerate}
\item 
 The graph $\Gamma_{R_n}$ is symmetric.
 \item 
If $n$ is odd, then $\Gamma_{R_n}$ is isomorphic to the complete graph $K_n$.
\item 
If $n$ is even, then $\Gamma_{R_n}$ has two connected components, each isomorphic to the complete graph $K_{m}$, where $m=\frac{n}{2}$.
\end{enumerate} 
\end{Proposition}
\begin{proof}
{\bf 1.} For the dihedral quandle $R_n$, the equality $a \rhd b = c$ means that 
$c = 2b - a  \pmod n$ or, equivalently, $c + a = 2b \pmod n$. 
If such an element $b$ exists, then 
$a = 2b - c  \pmod n$, that is, $c \rhd b = a$. 
Thus, the graph $\Gamma_{R_n}$ is symmetric.

{\bf 2.} If $n$ is odd, the element $2^{-1}$ exists, and hence for any 
$a, c \in \mathbb{Z}_n$ we can find $b$ such that $a \rhd b = c$. 
This means that any $a$ and $c$ are adjacent in $\Gamma_{R_n}$. 
Therefore, $\Gamma_{R_n}$ is isomorphic to the complete graph $K_n$.

{\bf 3.} If $n$ is even, a solution of the equation 
$c + a = 2b  \pmod n$ exists if and only if $c + a$ is even. 
Therefore, $\Gamma_{R_n}$ has two connected components, each of which is isomorphic to the complete graph $K_m$, where $m = \frac{n}{2}$.
\end{proof}

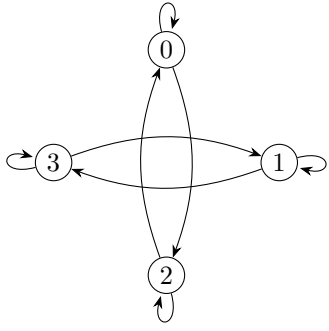
\begin{figure}[h]
\centering
\begin{tikzpicture}[>=Stealth,
  every node/.style={circle, draw, inner sep=2pt, font=\small}
]

\node (0) at (0, 1.5) {$0$};
\node (1) at (1.5, 0) {$1$};
\node (2) at (0,-1.5) {$2$};
\node (3) at (-1.5,0) {$3$};

\draw[->] (0) edge[loop above] (0);
\draw[->] (2) edge[loop below] (2);
\draw[->] (1) to[bend left=20] (3);
\draw[->] (3) to[bend left=20] (1);

\draw[->] (1) edge[loop right] (1);
\draw[->] (3) edge[loop left]  (3);
\draw[->] (0) to[bend left=20] (2);
\draw[->] (2) to[bend left=20] (0);

\end{tikzpicture}

\caption{Cayley graph of the dihedral quandle $R_4$.
}

\end{figure}

\begin{Theorem}
\label{alexanderiso}
 Let $G$ be a finite abelian group.  Then the connected components of the Cayley graph  $\Gamma_{A_t(G)}$ (of an Alexander quandle $A_t(G)$) are complete graphs induced by the cosets of $\im(\id-t)$. Consequently, $\Gamma_{A_t(G)}$ is a disjoint union of $\frac{|G|}{|\im(\id-t)|}$ copies of a complete graph on $|\im(\id-t)|$ vertices.
\end{Theorem}
\begin{proof}
    Choose any $a \in G$. We have to show that there is an edge from $a$ to $a+c$ for every $c \in \im(\id-t)$. However, for any $c \in \im(\id-t)$ there exists $d \in G$ such that $c=d-t(d)$, so $a+c=a+d-t(d)=t(a)+a+d-t(a+d)= t(a)+(\id-t)(a+d) = a \rhd (a+d)$, so, by definition of the Cayley graph, there is an edge from $a$ to $a+c$.
    
    On the other hand, suppose that there is an edge from $a$ to $x$ in $\Gamma_{A_t(G)}$.
    Then there exists $b \in G$ such that $x=a \rhd b=t(a)+b-t(b)$. Thus, $x=a+(\id-t)(b-a) \in a + \im(\id-t)$.
    
    Therefore, if $u,v$ are in the same coset $a+\im(\id-t)$, then 
$u-v \in \im(\id-t)$, so by the first part there is an edge 
$uv$ in $\Gamma_{A_t(G)}$. Hence, the induced subgraph on the coset is complete. Since every edge from a vertex $a$ ends in the coset $a+\im(\id-t)$, no directed path can leave this coset. Hence the connected components are precisely the cosets of $\im(\id-t)$.
\end{proof}

The next statement follows directly from  Theorem \ref{alexanderiso}.

\begin{Corollary}
    Let $G$ be a finite abelian group and let $t_1$  and $t_2$ be automorphisms of $G$. Then the Cayley graphs of $A_{t_1}(G)$ and $A_{t_2}(G)$ are isomorphic if and only if $|\im(\id-t_1)|=|\im(\id-t_2)|.$
\end{Corollary}

The next example shows that two different Alexander quandles $A_{t_1}(G)$ and $A_{t_2}(G)$ can have isomorphic Cayley graphs, but the subgroups $\im(\id-t_1)$ and $\im(\id-t_2)$ need not be isomorphic.

\begin{Example}
    Let $G=\ZZ_4 \times \ZZ_4$, $t_1(x,y)=(y,3x+2y)$, and $t_2(x,y)=(x+2y,2x+y)$. Observe that $\im(\id-t_1)=\{(0,0),(1,1),(2,2),(3,3)\}$ and $\im(\id-t_2)=\{(0,0),(2,2),(0,2),(2,0)\}$ are nonisomorphic as abelian groups. However, the Cayley graphs of the corresponding Alexander quandles are nonetheless isomorphic by Theorem \ref{alexanderiso}. 
    
    Note also that the two Alexander quandles themselves are also nonisomorphic, since $t_1^2=\id$ and $t_2^2 \neq \id$.
\end{Example}

\bigskip

 \section{The Cayley graph of a generalized Alexander quandle}
\bigskip

In this section, we study the Cayley graph of a generalized Alexander quandle.
We start with the following lemma.

\begin{Lemma}
\label{fixedpts}
   Let $G$ be a group, $\varphi \in \Aut(G)$ and let $A_\varphi(G)$ denote the generalized Alexander quandle of $G$ with respect to $\varphi$. Then for any $x \in G$, we have $x \rhd y = x \rhd z$ for some $y,z \in G$ if and only if $\varphi(yz^{-1})=yz^{-1}$.
\end{Lemma}
\begin{proof}
    Suppose we have $x \rhd y = x \rhd z$. Then $\varphi(xy^{-1})y=\varphi(xz^{-1})z$, so using the multiplicativity of $\varphi$, we get $\varphi(x)\varphi(y^{-1})y=\varphi(x)\varphi(z^{-1})z$. By canceling $\varphi(x)$ we arrive at $yz^{-1}=\varphi(y)\varphi(z^{-1})$, so $\varphi(yz^{-1})=yz^{-1}$.
    
    On the other hand, suppose $\varphi(yz^{-1})=yz^{-1}$. Then $\varphi(y^{-1})y=\varphi(z^{-1})z$ and thus also $\varphi(x)\varphi(y^{-1})y=\varphi(x)\varphi(z^{-1})z$ for every $x \in G$, which finally yields $x \rhd y = x \rhd z$.
\end{proof}

We can now state the following theorem.

\begin{Theorem}
\label{regular}
    Let $G$ be a group, $\varphi \in \Aut(G)$ and let $A_\varphi(G)$ denote the generalized Alexander quandle of $G$ with respect to $\varphi$. Then the Cayley graph $\Gamma_{A_\varphi(G)}$ is regular, and both the in-degree and the out-degree of each vertex are equal to the index in $G$ of the subgroup of fixed  points of $\varphi$.
\end{Theorem}
\begin{proof}
Denote $H=\{g\in G : \varphi(g)=g\}$ and note that $H$ is a subgroup of $G$.
Fix $x \in G$ and consider the map
\[
f_x : G \to G, \qquad f_x(y)=x \rhd y.
\]
By Lemma~\ref{fixedpts}, for $y,z\in G$ we have
\[
f_x(y)=f_x(z)
\quad\Longleftrightarrow\quad
yz^{-1}\in H.
\]
Thus the fibers of $f_x$ are exactly the left cosets of $H$. Hence the number of distinct values of $f_x$ equals the number of such cosets, namely
\[
d^+(x)=\bigl|\{x \rhd y : y\in G\}\bigr|=|\im(f_x)|=[G:H].
\]

Now fix $z \in G$. We compute
\[
d^-(z)
=
\bigl|\{x\in G : \exists\, y\in G \text{ with } x\rhd y=z\}\bigr|.
\]

The equation $x\rhd y=z$ is equivalent to $\varphi(xy^{-1})y=z$.
Multiplying on the right by $y^{-1}$ and applying $\varphi^{-1}$ yields
\[
xy^{-1}=\varphi^{-1}(zy^{-1}),
\]
and hence
\[
x=\varphi^{-1}(zy^{-1})\,y.
\]
Define $g_z(y)=\varphi^{-1}(zy^{-1})\,y$.
Then $x\rhd y=z$ if and only if $x=g_z(y)$, so $d^-(z)=|\im(g_z)|$.
Now observe that $g_z(y)$ differs from $f_{z}(y)=\varphi(zy^{-1})y$ only by replacing $\varphi$ with $\varphi^{-1}$. However, $\varphi(g)=g$ if and only if $\varphi^{-1}(g)=g$, so 
$H=\{g\in G : \varphi^{-1}(g)=g\}$, and therefore
the same fiber-counting argument as above shows that
\[
d^-(z)=|\im(g_z)|=[G:H].
\]

Since both in-degree and out-degree equal $[G:H]$ for every vertex, the graph is regular.
\end{proof}

As the next example shows, the connected components of the (regular) Cayley graph of a generalized Alexander quandle need not be complete graphs.

\begin{Example}
    Let $G=S_4$ and let $\varphi \in \Aut(G)$ be given by $\varphi(g)=(12)g(12)^{-1}$. Since $H=\{g \in S_4 : \varphi(g)=g\}$ is a subgroup of order $4$, Theorem \ref{regular} shows that $\Gamma_{A_\varphi(G)}$ is a regular graph in which every vertex has in-degree and out-degree equal to $6$. Figure \ref{s4pic} depicts the connected component of $\Gamma_{A_\varphi(G)}$ containing the identity permutation. 
\end{Example}

\begin{figure}[h]
\centering
\begin{tikzpicture}[
  >=Stealth,
  vtx/.style={circle,draw,minimum size=8mm,inner sep=1pt,font=\small},
  loopedge/.style={->,thin},
  biedge/.style={<->,thin}
]
\def\r{4.2}

\node[vtx] (v0)  at ( 90:\r) {$\mathrm{id}$};
\node[vtx] (v1)  at ( 60:\r) {$(123)$};
\node[vtx] (v2)  at ( 30:\r) {$(132)$};
\node[vtx] (v3)  at (  0:\r) {$(124)$};
\node[vtx] (v4)  at (-30:\r) {$(142)$};
\node[vtx] (v5)  at (-60:\r) {$(134)$};
\node[vtx] (v6)  at (-90:\r) {$(143)$};
\node[vtx] (v7)  at (-120:\r){$(234)$};
\node[vtx] (v8)  at (-150:\r){$(243)$};
\node[vtx] (v9)  at (180:\r) {$(12)(34)$};
\node[vtx] (v10) at (150:\r) {$(13)(24)$};
\node[vtx] (v11) at (120:\r) {$(14)(23)$};

\foreach \v in {v0,v1,v2,v3,v4,v5,v6,v7,v8,v9,v10,v11}{
  \draw[loopedge] (\v) edge[loop above, looseness=6] (\v);
}

\draw[biedge] (v0) -- (v1);
\draw[biedge] (v0) -- (v2);
\draw[biedge] (v0) -- (v3);
\draw[biedge] (v0) -- (v4);
\draw[biedge] (v0) -- (v9);

\draw[biedge] (v1) -- (v2);
\draw[biedge] (v1) -- (v7);
\draw[biedge] (v1) -- (v8);
\draw[biedge] (v1) -- (v11);

\draw[biedge] (v2) -- (v5);
\draw[biedge] (v2) -- (v6);
\draw[biedge] (v2) -- (v10);

\draw[biedge] (v3) -- (v4);
\draw[biedge] (v3) -- (v7);
\draw[biedge] (v3) -- (v8);
\draw[biedge] (v3) -- (v10);

\draw[biedge] (v4) -- (v5);
\draw[biedge] (v4) -- (v6);
\draw[biedge] (v4) -- (v11);

\draw[biedge] (v5) -- (v7);
\draw[biedge] (v5) -- (v9);
\draw[biedge] (v5) -- (v10);

\draw[biedge] (v6) -- (v8);
\draw[biedge] (v6) -- (v9);
\draw[biedge] (v6) -- (v11);

\draw[biedge] (v7) -- (v9);
\draw[biedge] (v7) -- (v11);

\draw[biedge] (v8) -- (v9);
\draw[biedge] (v8) -- (v10);

\draw[biedge] (v10) -- (v11);

\end{tikzpicture}
\caption{Cayley graph of the connected component containing the identity in the generalized Alexander quandle $(S_4,\varphi)$ with $\varphi(g)=(12)g(12)^{-1}$.}
\label{s4pic}
\end{figure}
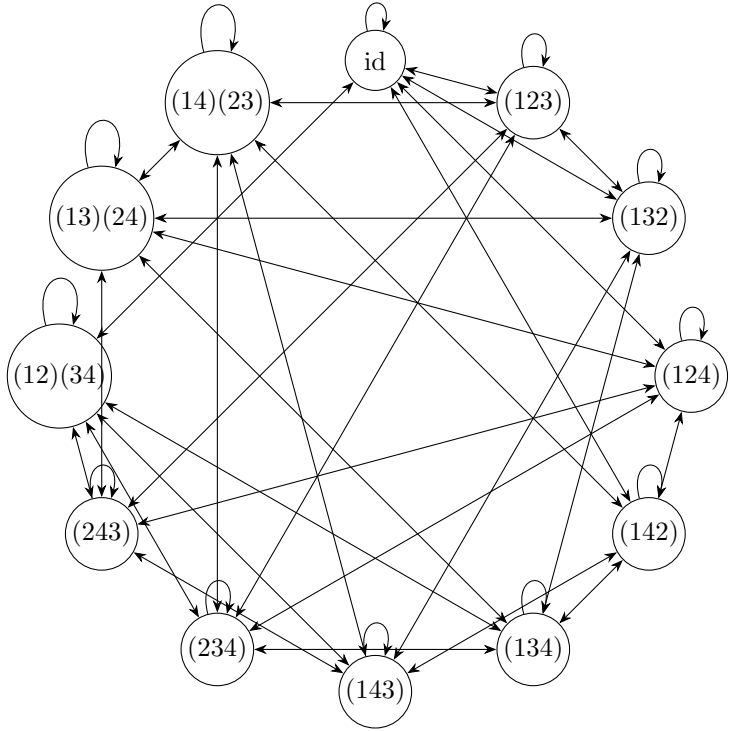

In all the examples considered above, the Cayley graph of a quandle is symmetric: for every edge $ab$, the edge $ba$ also appears. This naturally raises the question of whether the Cayley graph is symmetric for every quandle. The answer is negative, as shown by the following example.

The same example also demonstrates that the diameter of a connected component of the Cayley graph can be arbitrarily large.

\begin{Example}
Let $G = D_m = \langle r, s : r^m = e, s^2 = e, srs = r^{-1} \rangle$ be the dihedral group of order $2m$ and consider $\varphi \in \Aut(G)$ with $\varphi(g)=rgr^{-1}$. The operation in the generalized Alexander quandle $A_\varphi(D_m)$ can be written as
\[
a \rhd b = \varphi(ab^{-1})b = rab^{-1}r^{-1}b .
\]
We describe the edges of the Cayley graph $\Gamma_{A_\varphi(D_m)}$. By definition, there is an edge from $a$ to $a \rhd b$ for each $b \in D_m$. Recall that the elements of the dihedral group $D_m$ can be partitioned into the two subsets
\[
\{e, r, r^2, \ldots, r^{m-1}\}
\quad \text{and} \quad
\{s, rs, r^2s, \ldots, r^{m-1}s\}.
\]
Therefore, we consider the following cases.

{\bf 1.} Let $a = r^i$ and $b = r^j$, where $0 \le i,j \le m-1$. Then we obtain
\[
r^i \rhd r^j = r r^i r^{-j} r^{-1} r^j = r^i.
\]
Thus, the edges of this type are loops at the vertices $r^i$. 

{\bf 2.} Let $a = r^i$ and $b = r^j s$, where $0 \le i,j \le m-1$. Then we have
\[
r^i \rhd r^j s = r r^i s^{-1} r^{-j} r^{-1} r^j s
= r^{i+1} s r^{-1} s
= r^{i+2}.
\]
Hence there is an edge from $r^i$ to $r^{i+2}$.

Therefore, from each vertex of the form $r^i$, $0 \le i \le m-1$, exactly two edges originate: a loop and an edge to the vertex $r^{i+2}$.

If $m$ is even, then the graph $\Gamma_{A_\varphi(D_m)}$ contains two cycles of length $\frac{m}{2}$ with vertices
$ e, r^2, r^4, \ldots, r^{m-2} $ and $ r, r^3, \ldots, r^{m-1}.$
If $m$ is odd, then the graph $\Gamma_{A_\varphi(D_m)}$ contains one cycle of length $m$ with vertices $e, r^2, r^4, \ldots, r^{m-1}, r, r^3, \ldots, r^{m-2}.$

{\bf 3.} Let  $a = r^is$ and $b = r^j $, where $0 \le i,j \le m-1$. Then
\[
r^i s\rhd r^j = r r^i s r^{-j} r^{-1} r^j = r^{i+1} s r^{-1} = r^{i+2}s.
\]
Hence there is an edge from $r^is$ to $r^{i+2}s$.

{\bf 4.} Now let $a = r^is$ and $b = r^js $, where $0 \le i,j \le m-1$. Then
\[
r^i s\rhd r^js = r r^i s s^{-1} r^{-j} r^{-1} r^j s = r^{i+1} r^{-1} s= r^{i}s.
\]
Thus, the edges of this type are loops at the vertices $r^is$. 
Therefore, from each vertex of the form $r^i s$, $0 \le i \le m-1$, exactly two edges originate: a loop and an edge to the vertex $r^{i+2}s$.

If $m$ is even, then the graph $\Gamma_{A_\varphi(D_m)}$ contains two cycles of length $\frac{m}{2}$ with vertices
$ s, r^2s, r^4s, \ldots, r^{m-2}s $ and $ rs, r^3s, \ldots, r^{m-1}s.$
If $m$ is odd, then the graph $\Gamma_{A_\varphi(D_m)}$ contains one cycle of length $m$ with vertices $s, r^2s, r^4s, \ldots, r^{m-1}s, rs, r^3s, \ldots, r^{m-2}s.$

Note that the diameter of the connected components is equal to $\frac{m}{2}-1$ when $m$ is even, and $m-1$, when $m$ is odd. Figures~\ref{dihedraleven} and \ref{dihedralodd} illustrate these graphs.
\end{Example}

\begin{figure}[ht]
\centering
\begin{tikzpicture}[
  vertex/.style={circle,draw,minimum size=16pt,inner sep=1pt},
  arr/.style={-{Stealth[length=2.2mm,width=1.6mm]}, line width=0.5pt},
  dotseg/.style={densely dotted, line width=0.8pt},
  every node/.style={font=\small}
]

\newcommand{\LoopAbove}[1]{%
  \draw[arr]
    ([xshift= 2pt]#1.north)
    .. controls +(0.55,0.75) and +(-0.55,0.75) ..
    ([xshift=-2pt]#1.north);
}

\newcommand{\splitarc}[2]{%
  \path (#1) to[bend left=12] coordinate[pos=0.33] (p) coordinate[pos=0.67] (q) (#2);
  \draw[arr] (#1) to[bend left=12] (p);
  \draw[dotseg] (p) to[bend left=12] node[midway,fill=white,inner sep=1pt] {$\cdots$} (q);
  \draw[arr] (q) to[bend left=12] (#2);
}

\coordinate (C1) at (-4.5,  2.2); 
\coordinate (C2) at (-4.5, -3.0); 
\coordinate (C3) at ( 4.5,  2.2); 
\coordinate (C4) at ( 4.5, -3.0); 
\def\R{1.55}

\node[vertex] (e)   at ($(C1)+(180:\R)$) {$e$};
\node[vertex] (r2)  at ($(C1)+( 90:\R)$) {$r^{2}$};
\node[vertex] (r4)  at ($(C1)+(  0:\R)$) {$r^{4}$};
\node[vertex] (rm2) at ($(C1)+(-90:\R)$) {$r^{m-2}$};

\LoopAbove{e}  \LoopAbove{r2} \LoopAbove{r4} \LoopAbove{rm2}

\draw[arr] (e)  to[bend left=12] (r2);
\draw[arr] (r2) to[bend left=12] (r4);
\splitarc{r4}{rm2}
\draw[arr] (rm2) to[bend left=12] (e);

\node[vertex] (r)   at ($(C2)+(180:\R)$) {$r$};
\node[vertex] (r3)  at ($(C2)+( 90:\R)$) {$r^{3}$};
\node[vertex] (r5)  at ($(C2)+(  0:\R)$) {$r^{5}$};
\node[vertex] (rm1) at ($(C2)+(-90:\R)$) {$r^{m-1}$};

\LoopAbove{r} \LoopAbove{r3} \LoopAbove{r5} \LoopAbove{rm1}

\draw[arr] (r)  to[bend left=12] (r3);
\draw[arr] (r3) to[bend left=12] (r5);
\splitarc{r5}{rm1}
\draw[arr] (rm1) to[bend left=12] (r);

\node[vertex] (s)    at ($(C3)+(180:\R)$) {$s$};
\node[vertex] (r2s)  at ($(C3)+( 90:\R)$) {$r^{2}s$};
\node[vertex] (r4s)  at ($(C3)+(  0:\R)$) {$r^{4}s$};
\node[vertex] (rm2s) at ($(C3)+(-90:\R)$) {$r^{m-2}s$};

\LoopAbove{s} \LoopAbove{r2s} \LoopAbove{r4s} \LoopAbove{rm2s}

\draw[arr] (s)   to[bend left=12] (r2s);
\draw[arr] (r2s) to[bend left=12] (r4s);
\splitarc{r4s}{rm2s}
\draw[arr] (rm2s) to[bend left=12] (s);

\node[vertex] (rs)   at ($(C4)+(180:\R)$) {$rs$};
\node[vertex] (r3s)  at ($(C4)+( 90:\R)$) {$r^{3}s$};
\node[vertex] (r5s)  at ($(C4)+(  0:\R)$) {$r^{5}s$};
\node[vertex] (rm1s) at ($(C4)+(-90:\R)$) {$r^{m-1}s$};

\LoopAbove{rs} \LoopAbove{r3s} \LoopAbove{r5s} \LoopAbove{rm1s}

\draw[arr] (rs)  to[bend left=12] (r3s);
\draw[arr] (r3s) to[bend left=12] (r5s);
\splitarc{r5s}{rm1s}
\draw[arr] (rm1s) to[bend left=12] (rs);

\end{tikzpicture}

\caption{Structure of the Cayley graph of $A_\varphi(D_m)$ for even $m$ with
$\varphi(g)=rgr^{-1}$.}
\label{dihedraleven}
\end{figure}
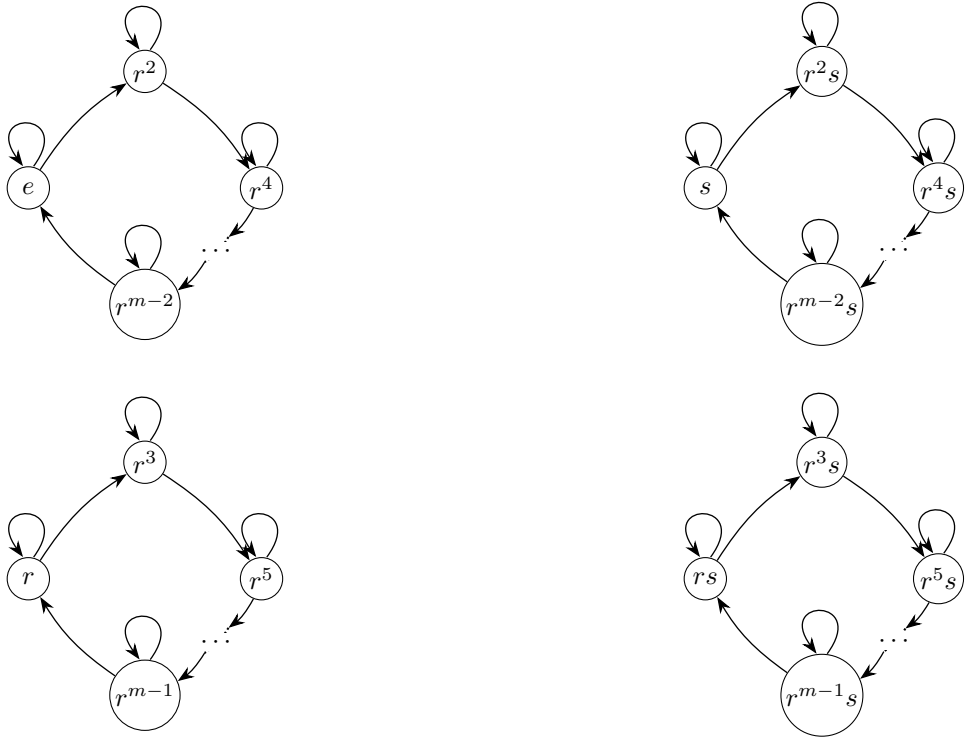

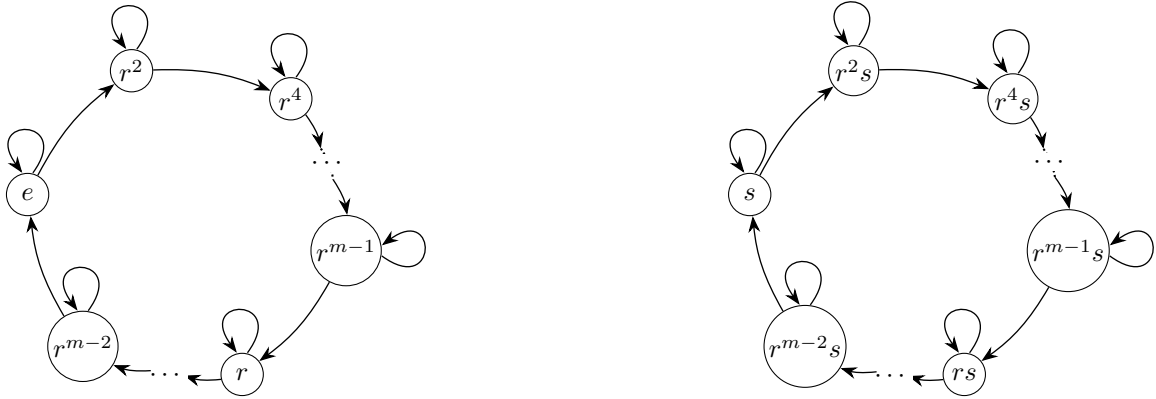
\begin{figure}[ht]
\centering
\begin{tikzpicture}[
  vertex/.style={circle,draw,minimum size=16pt,inner sep=1pt},
  arr/.style={-{Stealth[length=2.2mm,width=1.6mm]}, line width=0.5pt},
  dotseg/.style={densely dotted, line width=0.8pt},
  every node/.style={font=\small}
]

\newcommand{\LoopAbove}[1]{%
  \draw[arr]
    ([xshift= 2pt]#1.north) 
    .. controls +(0.55,0.75) and +(-0.55,0.75) ..
    ([xshift=-2pt]#1.north); 
}

\newcommand{\LoopRight}[1]{%
  \draw[arr]
    ([yshift=-2pt]#1.east) 
    .. controls +(0.75,-0.55) and +(0.75,0.55) ..
    ([yshift= 2pt]#1.east); 
}

\newcommand{\splitarc}[2]{%
  \path (#1) to[bend left=12] coordinate[pos=0.33] (p) coordinate[pos=0.67] (q) (#2);
  \draw[arr] (#1) to[bend left=12] (p);
  \draw[dotseg] (p) to[bend left=12] node[midway,fill=white,inner sep=1pt] {$\cdots$} (q);
  \draw[arr] (q) to[bend left=12] (#2);
}

\coordinate (C1) at (-4.8, 0.6);  
\coordinate (C2) at ( 4.8, 0.6);  
\def\R{2.15}

\node[vertex] (e)    at ($(C1)+(170:\R)$) {$e$};
\node[vertex] (r2)   at ($(C1)+(110:\R)$) {$r^{2}$};
\node[vertex] (r4)   at ($(C1)+( 50:\R)$) {$r^{4}$};
\node[vertex] (rm1)  at ($(C1)+(-10:\R)$) {$r^{m-1}$};
\node[vertex] (r)    at ($(C1)+(-70:\R)$) {$r$};
\node[vertex] (rm2)  at ($(C1)+(-130:\R)$) {$r^{m-2}$};

\LoopAbove{e}
\LoopAbove{r2}
\LoopAbove{r4}
\LoopRight{rm1}
\LoopAbove{r}
\LoopAbove{rm2}

\draw[arr] (e)  to[bend left=12] (r2);
\draw[arr] (r2) to[bend left=12] (r4);
\splitarc{r4}{rm1}
\draw[arr] (rm1) to[bend left=12] (r);
\splitarc{r}{rm2}
\draw[arr] (rm2) to[bend left=12] (e);

\node[vertex] (s)     at ($(C2)+(170:\R)$) {$s$};
\node[vertex] (r2s)   at ($(C2)+(110:\R)$) {$r^{2}s$};
\node[vertex] (r4s)   at ($(C2)+( 50:\R)$) {$r^{4}s$};
\node[vertex] (rm1s)  at ($(C2)+(-10:\R)$) {$r^{m-1}s$};
\node[vertex] (rs)    at ($(C2)+(-70:\R)$) {$rs$};
\node[vertex] (rm2s)  at ($(C2)+(-130:\R)$) {$r^{m-2}s$};

\LoopAbove{s}
\LoopAbove{r2s}
\LoopAbove{r4s}
\LoopRight{rm1s}
\LoopAbove{rs}
\LoopAbove{rm2s}

\draw[arr] (s)   to[bend left=12] (r2s);
\draw[arr] (r2s) to[bend left=12] (r4s);
\splitarc{r4s}{rm1s}
\draw[arr] (rm1s) to[bend left=12] (rs);
\splitarc{rs}{rm2s}
\draw[arr] (rm2s) to[bend left=12] (s);

\end{tikzpicture}

\caption{Structure of the Cayley graph of $A_\varphi(D_m)$ for odd $m$ with
$\varphi(g)=rgr^{-1}$.}
\label{dihedralodd}
\end{figure}

In the remainder of this paper, we analyze the structure of $A_\varphi(G)$ in the general case. 
We immediately have the following lemma.

\begin{Lemma}
Let $Q$ be a quandle. Then,  for every $b \in Q$, the right translation $R_b : Q \to Q$, defined by $R_b(x) = x \rhd b$, is a quandle automorphism of $Q$.
\end{Lemma}

\begin{proof}
Let $b \in Q$. By the right invertibility axiom of quandles, $R_b$ is a bijection.

Let $x,y \in Q$. Using the self-distributivity axiom, we obtain
\[
R_b(x \rhd y) = (x \rhd y) \rhd b = (x \rhd b) \rhd (y \rhd b) = R_b(x) \rhd R_b(y).
\]
Thus $R_b$ is a quandle homomorphism. Since it is also bijective, it follows that $R_b$ is a quandle automorphism.
\end{proof}

We shall need the following definition.

\begin{Def}
Let $Q$ be a quandle. For $x \in Q$ the \emph{forward orbit} of $x$ is
\[
\mathcal O^{+}(x)
  = \{ R_{y_k}\cdots R_{y_1}(x) : k \ge 0,\; y_1,\ldots,y_k \in Q \}.
\]

Furthermore, the \emph{inner automorphism group} of $Q$ is
\[
\Inn(Q)=\langle R_b : b \in Q\rangle \le \Aut(Q),
\]
the subgroup of $\Aut(Q)$ generated by all right translations.
\end{Def}

The following lemma follows easily from the definition and the finiteness of $Q$.

\begin{Lemma}
\label{inng}
Let $Q$ be a finite quandle and $x \in Q$. Then
\[
\mathcal O^+(x)=\{g(x) : g \in \Inn(Q)\}.
\]
\end{Lemma}

\begin{proof}
By definition, $\mathcal O^+(x)$ consists of all elements obtained from $x$ by compositions of right translations $R_b$. Hence
\[
\mathcal O^+(x) \subseteq \{g(x):g\in\Inn(Q)\}.
\]

Since $Q$ is finite, each $R_b$ is a permutation of a finite set and therefore has finite order. Hence $R_b^{-1}=R_b^{k}$ for some $k>0$. Consequently every element of $\Inn(Q)$ can be written as a composition of positive powers of right translations. Applying such a composition to $x$ yields an element of $\mathcal O^+(x)$.

Therefore $\{g(x):g\in\Inn(Q)\}\subseteq \mathcal O^+(x)$, and the two sets coincide.
\end{proof}

When the automorphism in the definition of the generalized Alexander quandle is inner, the structure becomes more tractable. The following theorem plays a key role in our analysis of this case.

\begin{Theorem}
\label{orbits}
Let $G$ be a finite group and $h\in G$. Let $\varphi:G\to G$ be the inner automorphism given by
$\varphi(g)=hgh^{-1}$ and consider the generalized Alexander quandle $A_\varphi(G)$.
Let
\[
N=\Big\langle\,[h,g]\ \Big|\ g\in G\,\Big\rangle
\]
be the subgroup generated by the commutators
$[h,g]=hgh^{-1}g^{-1}$.
Then 
\[
\mathcal O^{+}(x)=xN \qquad \text{for all } x\in G.
\]
\end{Theorem}

\begin{proof}
Note that we have $[h,g]^{-1}=[g,h]$ for all $g\in G$.
For $x\in G$ we have
\[
x\,[h,x^{-1}]^{-1} \;=\; x\,(x^{-1}hx h^{-1}) \;=\; hxh^{-1}.
\]
Therefore, for all $x,y\in G$,
\[
x\rhd y
= hxy^{-1}h^{-1}y
= (hxh^{-1})\,(hy^{-1}h^{-1}y)
= x\,[h,x^{-1}]^{-1}\,[h,y^{-1}] \in xN.
\]
Thus every right translation, as well as any composition of right translations, maps $xN$ into itself.  Hence
\[
\mathcal O^{+}(x) \subseteq xN.
\]

Let us now prove that we also have  $xN \subseteq \mathcal O^{+}(x)$ for every $x \in G$.
Denote the identity of $G$ by $e$. Then
\[
R_e(x)=x\rhd e=hxe^{-1}h^{-1}e=hxh^{-1},
\]
so $R_e$ is a permutation of the finite set $G$. Let $m\ge 1$ be the (finite) order of $R_e$ in
$\mathrm{Sym}(G)$, so $R_e^m=\mathrm{id}$ and hence $R_e^{-1}=R_e^{m-1}$.
Fix $u\in G$ and consider the map
\[
T_u := R_u\,R_e^{m-1} \in \Inn(A_\varphi(G)).
\]
For any $x\in G$, we have
\[
T_u(x)=R_u(R_e^{-1}(x)).
\]
But $R_e^{-1}(x)=h^{-1}xh$, and therefore
\[
T_u(x)
= (h^{-1}xh)\rhd u
= h(h^{-1}xh)u^{-1}h^{-1}u
= x\,(hu^{-1}h^{-1}u)
= x\,[h,u^{-1}].
\]
So $T_u$ acts as right multiplication by the commutator $[h,u^{-1}]$.
However, $\{u^{-1} : u\in G\}=G$, so the elements $[h,u^{-1}]$
($u\in G$) generate $N$. Consequently, for every $n\in N$ there exist $u_1,\dots,u_r\in G$ and
$\varepsilon_1,\dots,\varepsilon_r\in\{\pm1\}$ such that
\[
n=[h,u_1^{-1}]^{\varepsilon_1}\cdots [h,u_r^{-1}]^{\varepsilon_r}.
\]
Hence $xn=T_{u_r}^{\varepsilon_r}\ldots T_{u_2}^{\varepsilon_2}T_{u_1}^{\varepsilon_1}(x)$. 
Therefore, the right-multiplication map $\rho_n:x\mapsto xn$ lies in $\Inn(A_{\varphi}(G))$, so
$xN \subseteq \{g(x) : g \in \Inn(A_{\varphi}(G))\}=\mathcal O^+(x)$ by Lemma \ref{inng}.
This proves that $\mathcal O^{+}(x)=xN$.
\end{proof}

We have the following lemma.

\begin{Lemma}
\label{normal}
The subgroup $N=\langle [h,g] : g\in G\rangle$ is normal in $G$.
\end{Lemma}
\begin{proof}
 Fix $x,g\in G$. Then
\[
x[g,h]x^{-1}=xghg^{-1}h^{-1}x^{-1}=(xg)h(xg)^{-1}\,(xh^{-1}x^{-1}).
\]
Insert $h^{-1}h$ between the last two factors to obtain
\[
x[g,h]x^{-1}=[xg,h] \, [h,x]=[h,xg]^{-1} \, [h,x] \in N.
\]
Since the elements $[g,h]=[h,g]^{-1}$ generate $N$, it follows that
$N\trianglelefteq G$.
\end{proof}

We now obtain the following corollary.

\begin{Corollary}
Let $G$ be a finite group and $h\in G$. Let $\varphi(g)=hgh^{-1}$ and consider the generalized Alexander quandle $A_\varphi(G)$. Let $N=\langle [h,g] : g\in G\rangle$.
Then the connected components of $\Gamma_{A_\varphi(G)}$ are precisely the cosets of $N$ in $G$. In particular, $\Gamma_{A_\varphi(G)}$ has $|G:N|$ connected components, each of size $|N|$.
Moreover, the induced subgraphs on the connected components are all isomorphic.
\end{Corollary}

\begin{proof}
By Theorem \ref{orbits}, the forward orbit of $x$ is $xN$.  Hence the connected components of $\Gamma_{A_\varphi(G)}$ are precisely the sets $xN$, and their number is $|G:N|$.

Fix $g\in G$ and define $\tau_g:N\to gN$ by $\tau_g(x)=xg$. Observe that since $N$ is normal in $G$ by Lemma \ref{normal}, we have $Ng=gN$, so indeed we have $\tau_g(x) \in gN$. It is clear that $\tau_g$ is a bijection.
Suppose that $xy$ is an edge in the induced subgraph on $N$. Therefore, there exists $z \in G$ such that $y=x \rhd z=hxz^{-1}h^{-1}z$. But then, $\tau_g(y)=yg=hxz^{-1}h^{-1}zg=h\tau_g(x)(zg)^{-1}h^{-1}(zg)$. So, denote $w=zg$ and observe that $\tau_g(y)=\tau_g(x) \rhd w$, therefore we have an edge from $\tau_g(x)$ to $\tau_g(y)$ in the induced subgraph on $gN$. 
So, we have proved that $\tau_g$ is a graph isomorphism between the induced subgraphs on $N$ and $gN$.
\end{proof}

\bigskip

{\bf Statements and Declarations} \\

The authors state that there are no competing interests. 

\bigskip

\bigskip

\bibliographystyle{amsplain}
\bibliography{biblio}

@article{Joyce82,
    AUTHOR = {Joyce, David},
     TITLE = {A classifying invariant of knots, the knot quandle},
   JOURNAL = {J. Pure Appl. Algebra},
  FJOURNAL = {Journal of Pure and Applied Algebra},
    VOLUME = {23},
      YEAR = {1982},
    NUMBER = {1},
     PAGES = {37--65},
      ISSN = {0022-4049,1873-1376},
   MRCLASS = {57M25 (20F29 20N05 53C35)},
  MRNUMBER = {638121},
MRREVIEWER = {Mark\ E.\ Kidwell},
       DOI = {10.1016/0022-4049(82)90077-9},
       URL = {https://doi.org/10.1016/0022-4049(82)90077-9},
}

@article{Matveev82,
    AUTHOR = {Matveev, S. V.},
     TITLE = {Distributive groupoids in knot theory},
   JOURNAL = {Mat. Sb. (N.S.)},
  FJOURNAL = {Matematicheski\u i\ Sbornik. Novaya Seriya},
    VOLUME = {119(161)},
      YEAR = {1982},
    NUMBER = {1},
     PAGES = {78--88, 160},
      ISSN = {0368-8666},
   MRCLASS = {57M25 (20L15)},
  MRNUMBER = {672410},
MRREVIEWER = {Jonathan\ A.\ Hillman},
}

@article{CarterEtAl99,
  author  = {Carter, J. Scott and Jelsovsky, Daniel and Kamada, Seiichi and Langford, Laurel and Saito, Masahico},
  title   = {Quandle cohomology and state-sum invariants of knotted curves and surfaces},
  journal = {Transactions of the American Mathematical Society},
  volume  = {355},
  number  = {10},
  pages   = {3947--3989},
  year    = {2003}
}

@book {CarterEtAl03,
    AUTHOR = {Carter, Scott and Kamada, Seiichi and Saito, Masahico},
     TITLE = {Surfaces in 4-space},
    SERIES = {Encyclopaedia of Mathematical Sciences},
    VOLUME = {142},
      NOTE = {Low-Dimensional Topology, III},
 PUBLISHER = {Springer-Verlag, Berlin},
      YEAR = {2004},
     PAGES = {xiv+213},
      ISBN = {3-540-21040-7},
   MRCLASS = {57Q45 (57M25 57R40)},
  MRNUMBER = {2060067},
MRREVIEWER = {Sergej\ V.\ Matveev},
       DOI = {10.1007/978-3-662-10162-9},
       URL = {https://doi.org/10.1007/978-3-662-10162-9},
}

@incollection {Kamanda17,
    AUTHOR = {Kamada, Seiichi},
     TITLE = {Knot invariants derived from quandles and racks},
 BOOKTITLE = {Invariants of knots and 3-manifolds ({K}yoto, 2001)},
    SERIES = {Geom. Topol. Monogr.},
    VOLUME = {4},
     PAGES = {103--117},
 PUBLISHER = {Geom. Topol. Publ., Coventry},
      YEAR = {2002},
   MRCLASS = {57M27},
  MRNUMBER = {2002606},
       DOI = {10.2140/gtm.2002.4.103},
       URL = {https://doi.org/10.2140/gtm.2002.4.103},
}

@article{Cayley78,
  author  = {Cayley, Arthur},
  title   = {Desiderata and Suggestions: No. 2. The Theory of Groups: Graphical Representation},
  journal = {American Journal of Mathematics},
  volume  = {1},
  number  = {2},
  pages   = {174--176},
  year    = {1878}
}

@book{GodsilRoyle01,
    AUTHOR = {Godsil, Chris and Royle, Gordon},
     TITLE = {Algebraic graph theory},
    SERIES = {Graduate Texts in Mathematics},
    VOLUME = {207},
 PUBLISHER = {Springer-Verlag, New York},
      YEAR = {2001},
     PAGES = {xx+439},
      ISBN = {0-387-95241-1; 0-387-95220-9},
   MRCLASS = {05-02 (05C50 05E30)},
  MRNUMBER = {1829620},
MRREVIEWER = {Robin\ J.\ Wilson},
       DOI = {10.1007/978-1-4613-0163-9},
       URL = {https://doi.org/10.1007/978-1-4613-0163-9},
}

@book{deLaHarpe00,
    AUTHOR = {de la Harpe, Pierre},
     TITLE = {Topics in geometric group theory},
    SERIES = {Chicago Lectures in Mathematics},
 PUBLISHER = {University of Chicago Press, Chicago, IL},
      YEAR = {2000},
     PAGES = {vi+310},
      ISBN = {0-226-31719-6; 0-226-31721-8},
   MRCLASS = {20F65 (20F69 57M07)},
  MRNUMBER = {1786869},
MRREVIEWER = {Lee\ Mosher},
}

@book{ElhamdadiNelson15,
    AUTHOR = {Elhamdadi, Mohamed and Nelson, Sam},
     TITLE = {Quandles---an introduction to the algebra of knots},
    SERIES = {Student Mathematical Library},
    VOLUME = {74},
 PUBLISHER = {American Mathematical Society, Providence, RI},
      YEAR = {2015},
     PAGES = {x+245},
      ISBN = {978-1-4704-2213-4},
   MRCLASS = {57M27 (57M25 57Q45)},
  MRNUMBER = {3379534},
MRREVIEWER = {Frederick\ Norwood},
       DOI = {10.1090/stml/074},
       URL = {https://doi.org/10.1090/stml/074},
}

@book{Diestel,
 author = {Diestel, Reinhard},
 title = {Graph theory},
 edition = {3rd revised and updated ed.},
 fseries = {Graduate Texts in Mathematics},
 series = {Grad. Texts Math.},
 issn = {0072-5285},
 volume = {173},
 isbn = {3-540-26182-6},
 year = {2005},
 publisher = {Berlin: Springer},
 }

@article{LucTa,
  author  = {Ta, Luc},
  title   = {Generalized Cayley graphs of racks and right quasigroups},
  journal = {arXiv:2506.04437},
  year    = {2025}
}

\bigskip

\end{document}